\documentclass[leqno,11pt,oneside]{amsart}

 \usepackage[english]{babel}

\usepackage{xcolor}
\usepackage{graphicx}
\usepackage{relsize}

\usepackage{enumitem}
\usepackage{amsmath,amsfonts,amssymb,amsthm}
\usepackage{soul}

\newcommand{\sing}{{\rm Sing}}

\newtheorem*{theorem*}{Theorem}

\newtheorem{teo}{Theorem}[section]
\newtheorem{prop}[teo]{Proposition}
\newtheorem{ddef}[teo]{Definition}

\newtheorem{lem}[teo]{Lemma}
\newtheorem*{cor*}{Corollary}

\newtheorem*{lem*}{Lemma}

\newtheorem*{teorA'}{Theorem A'}

\newtheorem*{fact*}{Fact}

\newtheorem{maintheorem}{Theorem}

\theoremstyle{definition}

\newcommand{\red}{\scriptscriptstyle\textup{red}}

\newcommand{\C}{\mathbb{C}}
 
\newcommand{\R}{\mathbb{R}}

\newcommand{\N}{\mathbb{N}}
\newcommand{\Z}{\mathbb{Z}}

\newcommand{\K}{\mathbb{K}}

\newcommand{\dd}{\textup{d}}

\newcommand{\F}{\mathcal{F}}

\newcommand{\G}{\mathcal{G}}

\newcommand{\im}{{\rm Im}}
\newcommand{\re}{{\rm Re}}

\newcommand{\mb}[1]{\mathbf{#1}}

\newcommand*\xbar[1]{ %
   \hbox{ %
     \vbox{%
       \hrule height 0.3pt 
       \kern0.35ex
       \hbox{%
         \kern-0.1em
         \ensuremath{#1}%
         \kern-0.1em
       }%
     }%
   }%
}

\newcommand*\xxbar[1]{%
   \hbox{%
     \vbox{%
       \hrule height 0.3pt 
       \kern0.4ex
       \hbox{%
         \kern-0.1em
         \ensuremath{#1}%
         \kern-0.1em
       }%
     }%
   }%
}

\makeindex

\begin{document}

\setcounter{section}{0}
\setcounter{teo}{0}
\setcounter{exe}{0}

\title{On the algebraicity of germs of meromorphic functions}

\author{Yohann Genzmer}
 \address{Institut de Math\'ematiques de Toulouse --- Universit\'e Paul Sabatier}
\curraddr{118 Route de Narbonne --- 31400
  --- Toulouse Cedex 9, FRANCE}
\email{yohann.genzmer@math.univ-toulouse.fr }

\author{Rog\'erio Mol}
\address{Departamento de Matem\'atica ---  Universidade Federal de Minas Gerais}
\curraddr{Av. Ant\^onio Carlos 6627 --- 31270-901 --- Belo Horizonte, BRAZIL.}
\email{rmol@ufmg.br}

\subjclass[2010]{32A20, 13B40,  14B12, 13J05, 13J15, 32S65, 34Cxx}
\keywords{Meromorphic function, algebraic function, homomorphic foliation, holomorphic vector field, Artin approximation}
\thanks{The second  author was financed in part by the
Coordena\c c\~ao de Aperfei\c coamento de Pessoal de N\'\i vel Superior - Brasil (CAPES) --- Finance Code 001}
\maketitle

\begin{abstract}
In this article we prove that every germ of analytic meromorphic function at $(\C^{2},0)$ is equivalent,
under the right composition by a germ of biholomorphism, to a germ of algebraic meromorphic function.
An analogous result is also true for real analytic meromorphic functions.
\end{abstract}


 \medskip \medskip

\section{Introduction}


Let $\varphi$ be a germ of holomorphic function at $0 \in \C^{n}$.
A well-known result by J. Mather and S. Yau asserts that
if $\varphi$ has isolated singularity, then it is \emph{polynomially determined}, that is, it is  equivalent, under the
right composition by a  germ of
biholomorphism in ${\rm Diff}(\C^{n},0)$, to a polynomial function \cite{MatherYau1981,MatherYau1982}.
More than that, $\varphi$ is \emph{finitely determined}, this polynomial being the initial part of its Taylor
series.
Using techniques of deformation of singular holomorphic foliations,  in \cite{cerveaumattei1982}, D. Cerveau and J.-F. Mattei address
 the problem of finite determination of germs  $\varphi = \varphi_{1}^{\mu_{1}} \cdots \varphi_{k}^{\mu_{k}}$
of multiform functions,
where $\varphi_{1},\ldots,\varphi_{k}$ are irreducible germs of holomorphic functions at $(\C^{n},0)$
 and $\mu_{1}, \ldots, \mu_{k} \in \C^{*}$. They prove  that
if the critical set $\mathrm{C}'(\varphi)$ of $\varphi$, as defined below, is reduced to at most an isolated point,
then  $\varphi$  is finitely determined.

A softer alternative to finite determination is the notion of \emph{algebraization}. We say that a germ of
meromorphic function at $(\K^{n},0)$, where $\K =\mathbb{R}$ or $\mathbb{C}$, is \emph{algebraizable} if it becomes algebraic after the right composition by
an element of ${\rm Diff}(\K^{2},0)$, the group of germs of
of real or complex analytic diffeomorphism at $0 \in \K^{2}$. A meromorphic function $\varphi$ in the variables $x=(x_{1},\cdots,x_{n})$ is \emph{algebraic} if
there exists a polynomial $P = P(x,t) \in \K[x][t]$ such that $P(x,\varphi(x)) \equiv 0$.
The central result of this article has the following statement:

\begin{maintheorem}
\label{theo-main}
Let $\K  =\mathbb{R}$ or $\mathbb{C}$. Let $\varphi$ be a germ of meromorphic function at $(\K^{2},0)$.
Then  there exists a germ of diffeomorphism $\Psi \in {\rm Diff}(\K^{2},0)$ such that
$\varphi \circ \Psi$ is algebraic.
\end{maintheorem}
We remark that this result is false in higher dimension: in \cite{Bochnak}, J. Bochnak and W. Kucharz provide  an example of a non-algebraizable holomorphic function in dimension three.

The proof of Theorem \ref{theo-main} has two key ingredients. First,    a sophisticated generalization of the Artin-Ploski-Popescu Approximation Theorem established in \cite{bilski2017}   by M. Bilski, A. Parunsinski and G. Rond. Second, Cerveau-Mattei's finite determination techniques, based upon the trivialization of
holomorphic foliations.

Let us give an outline on how these pieces  fit together in our proof.
Cerveau-Mattei's approach  \cite{cerveaumattei1982} relies on   regarding the holomorphic foliation defined by the levels of a multiform function $\varphi$  as above. The \emph{strict critical set} $\mathrm{C}'(\varphi)$ is then defined as the   singular set of the holomorphic $1$-form
\[ \theta_{\varphi} = \varphi_{1} \ldots \varphi_{k} \sum_{j=1}^{k} \mu_{j} \frac{\dd\varphi_{j}}{\varphi_{j}} ,\]
 excluding those points of the hypersurface
$\varphi_{1} \cdots \varphi_{k}=0$ where it   has normal crossings.
In dimension two, in general, the critical set $\mathrm{C}'(\varphi)$ coincides with the singular locus of  $\theta_{\varphi}$.
In the holomorphic case --- i.e. $\mu_{i} \in \mathbb{N}^{*}$ for $i=1,\ldots,k$ ---  the set  $\mathrm{C}'(\varphi)$ is at most a point,
and therefore, applying Cerveau-Mattei's result, $\varphi$ is of finite determination.
On the other hand, when $\mu_{i} \in \Z^{*}$ for $i=1,\ldots,k$, not all with the same sign, we are
 in the \emph{purely meromorphic} case,
  $\mathrm{C}'(\varphi)$ may be  one-dimensional and the result of \cite{cerveaumattei1982} no longer applies.

When  $\sing(\theta_{\varphi})$ has one dimensional components,  which must be contained in level sets of $\varphi$,
we can regard them, along with the irreducible equations of zeros and poles of $\varphi$,  as solutions of polynomial equations. The Artin-P\l oski-Popescu Approximation Theorem (\cite{bilski2017}, see Theorem \ref{teo-ploski} below) then provides algebraic solutions for the same set
of polynomial equations,  but  depending  on some additional variables. This allows us to make a deformation of
the equations of the irreducible components of $\sing(\theta_{\varphi})$, along with those of the zero and pole sets of $\varphi$, into
algebraic equations,
as tangent to the identity as wished.
Finally, Theorem \ref{theo-main} is obtained from a result in \cite{cerveaumattei1982} on the
triviality of deformations of holomorphic foliations.
The above steps, to be detailed in Section \ref{sec-main-theorem}, provide a  proof for Theorem \ref{theo-main}
in the complex case. The  real analytic case, to be treated in Section \ref{section-realanalytic}, will be proved through complexification,  after using   the fact that
the  Approximation Theorem also holds for real analytic functions.

This article can also be regarded under the perspective  of the \emph{algebraization} of germs of analytic differential
equations in two variables. In a more classical and strict sense, a germ of holomorphic foliation is \emph{algebraizable}
if there are local coordinates where it is engendered by a vector field with polynomial
coefficients, which allows to see it as the germification at a point
of a global foliation defined in the complex projective plane $\mathbb{P}^{2}$.
Whether or not all germs of holomorphic foliations are algebraizable was question dating back to the early years of the theory
of holomorphic foliation
\cite{dulac1909} that received a negative answer
in \cite{GenzmerTeyssier2010}, where the authors proved the existence of countably many non-algebraizable saddle-node equations ---
i.e. corresponding to a vector field with non-nilpotent linear part having one eigenvalue equal to zero.
In a  more fluid formulation,
a germ of holomorphic foliation $\F$ at $(\C^{2},0)$ is \emph{algebraizable}
if there exists an algebraic surface $S$ endowed with a singular holomorphic foliation $\G$
such that $\F$ is, under
a germ of biholomorphic identification of $(\C^{2},0)$ and $(S,p)$, for some $p \in S$,  the germ of $\G$ at
$p$. The articles \cite{Casale2013} and \cite{CalsamigliaSad2015}
study the algebraization of foliations
given by the level sets of  meromorphic
functions, under the condition that they become regular after a single blow-up.
The algebraization of a meromorphic function, as provided in the conclusion of   Theorem \ref{theo-main}, is  in some sense  the algebraization of the foliation defined by
its levels, as we will describe in the next section.

\par \noindent \textbf{Acknowledgments.}
The second author is grateful for the hospitality during his visit to the Institut de Ma\-th\'e\-ma\-ti\-ques
de Toulouse/Universit\'e Paul Sabatier  in the winter 2022/23, when this work was developed.
Both authors express  gratitude to Guillaume Rond for having drawn their attention to the Artin-Ploski-Popescu Approximation Theorem \cite{bilski2017}.

\section{Algebraic functions, holomorphic foliations}

Let us consider variables $x = (x_{1},\ldots,x_{n}) \in \C^{n}$.
Denote by $\C\{x\}$ the ring of convergent power series at $0 \in \C^{n}$  whose elements will  be called \emph{analytic functions}.

\begin{ddef}
An element $f \in  \C \{x\}$ is \emph{algebraic} if there is a polynomial $P \in \C[x][t]$, where $t \in \C$, such that $P(f) = 0$. Algebraic analytic functions form a  ring that will be denoted by $\C\langle x \rangle$.
\end{ddef}

In line with the definition above, we say that a germ of meromorphic function $\varphi$ at $(\C^{n},0)$ is \emph{algebraic}
if there exists a polynomial $P = P(x,t) \in \C[x][t]$, where $t \in \C$, such that $P(\varphi) = 0$.
This is equivalent to asking that $\varphi$ is the quotient $f/g$ for some $f,g \in \C\langle x \rangle$. In other words,
the field of algebraic meromorphic functions is the fraction field of the ring $\C\langle x \rangle$.
The polynomial $P \in \C[x][t]$ can be chosen to be minimal in degree and with   coefficients  relatively prime, in which case it will be called \emph{minimal}.

Our results are developed under the scope of the local theory of holomorphic foliations.
A germ of \emph{singular holomorphic foliation} of codimension one at $(\C^{n},0)$, $n \geq 2$, is the object $\F$
defined by a germ of holomorphic $1$-form $\omega$ at $(\C^{n},0)$, satisfying $\omega \wedge d \omega = 0$
(the integrability condition), with singular set $\sing(\omega)$ of codimension at least two.
We will refer to this object simply as \emph{foliation}.
The definition means that, in a neighborhood of $0 \in \C^{n}$ where the germ $\omega$ is realized, it defines
a regular holomorphic foliation outside $\sing(\omega)$.
A particular case, which is considered in this article,
is when $\F$ is defined by a closed meromorphic $1$-form $\lambda$.
The corresponding holomorphic $\omega$, the one that defines the foliation, is
obtained by multiplying $\lambda$ by an appropriated meromorphic function that cancels its zeros and its poles.
A simple and standard calculation shows  that the codimension one analytic hypersurfaces defined by the poles and zeros of $\lambda$
are invariant by $\F$.
As a fundamental example, to be treated in this article,
the levels of the multiform function
$\varphi = \varphi_{1}^{\mu_{1}} \cdots \varphi_{k}^{\mu_{k}}$, as in the Introduction,
define a foliation $\F$ given by the closed meromorphic $1$-form of \emph{logarithmic type}
$\lambda_{\varphi} = \frac{\dd \varphi}{\varphi} =  \sum_{j=1}^{k} \mu_{i} \frac{\dd\varphi_{j}}{\varphi_{j}}$.
The analytic hypersurfaces $\varphi_{j} =0$, $j=1,\ldots,k$,   the irreducible components of the polar set of $\lambda_{\varphi}$, are invariant by $\F$.
Cancelling the poles, we obtain a holomorphic $1$-form
$\theta_{\varphi} = \varphi_{1} \ldots \varphi_{k} \lambda_{\varphi}$.  The codimension one components of
$\sing(\theta_{\varphi})$ are also invariant by $\F$ and, clearly, are not zeros of $\varphi_{1} \cdots \varphi_{k}$.  We finally obtain the holomorphic $1$-form
 $\omega_{\varphi}$, that defines $\F$, by cancelling the common irreducible factors of the coefficients of $\theta_{\varphi}$.
These    elementary facts provide the first move towards  the proof of Theorem \ref{theo-main}, where
$\varphi$ is meromorphic in dimension $n=2$.

As we pointed out in the Introduction, the algebraization of a meromorphic function is
somehow linked to the notion of algebraization of a  foliation. Let us explain the
meaning of this.
Let $\varphi$ be a purely meromorphic algebraic function with minimal polynomial $P \in \C[x][t]$ and
denote by $\F$ the local holomorphic foliation
defined by its levels.
If  $\varphi$ is rational,   $\F$  is evidently the
germification of a holomorphic  foliation in $\mathbb{P}^{n}$.
Let us then assume that $\varphi$ is    non-rational, which is equivalent to asking
that $\deg(P) > 1$, where the degree is in the variable $t$.
In $\C^{n} \times \C  \cong \C^{n+1}$, consider the irreducible algebraic hypersurface $M$ of equation $P(x,t)= 0$.
Since $\varphi$ is purely meromorphic, for every $c \in \C$ the level hypersurface $\varphi = c$ contains the origin,
and thus $P(0,c) = 0$, following that $P(0,t) \equiv 0$ and thus the line $L$ of equation $x=0$ is contained in $M$.
Since $M$ is irreducible of codimension one, $L$ is accumulated by points  $M \setminus L$.
We can consider $(x,t)$ as affine coordinates of $\mathbb{P}^{n}$ and the corresponding
compactification of $M$.  The point at infinity of $L$ will be identified with $c = \infty$.
$M$ is endowed with a singular foliation $\G$
defined by the levels of the rational   function on $M$ given    by $R(x,t)=t$.
On the other hand, we can embed a small punctured
neighborhood $U^{*}$ of $0 \in \C^{2}$ in a neighborhood of $L$ in $M$, minus $L$ itself, by the correspondence
$x \in U^{*} \mapsto (x,\varphi(x)) \in M$. This map sends $\F$ into $\G$, meaning that the
levels of $\varphi$ are sent into the levels of $R$.
In other words, under this identification, $\phi$ is the rational function $R$.

 \section{Proof of Theorem \ref{theo-main}}
\label{sec-main-theorem}

In this section we provide a proof for Theorem \ref{theo-main} for $\K = \C$.
Our result involves approximating analytic functions which are solution of a system of polynomial equations by algebraic ones.
Our main tool is  Artin-P\l oski-Popescu Approximation Theorem
\cite[Th.2.1]{bilski2017}:

\begin{teo}
\label{teo-ploski}
Suppose that a system of polynomial equations in the variables $y \in \C^{m}$,
$$\mb{f} =  (f_{1},\ldots,f_{\gamma}) \in \C\langle x\rangle[y]^{\gamma},$$
where $x \in \R^{n}$ and $\gamma \in \mathbb{N}$,
has a   vector of analytic functions $\mb{y}(x)  \in \C\{x\}^{m}$ as solution, that is
\[\mb{f}(x,\mb{y}(x)) = 0.\]
Then there are variables $z = (z_{1},...,z_{\sigma}) \in \C^{\sigma}$,
a vector of algebraic functions $\mb{\widehat{y}}(x,z) \in \C \langle x,z \rangle^{m}$
and a vector of analytic functions
$\mb{z}(x) \in \C\{x\}^{\sigma}$, with $\mb{z}(0) = 0$,
such that
\[  \mb{f}(x,\mb{\widehat{y}}(x,z)) = 0  \ \ \ \text{and} \ \ \  \mb{y}(x) = \mb{\widehat{y}}(x,\mb{z}(x)) .\]
\end{teo}
In other words, if a system of polynomial equations    has an analytic solution, then it has an
algebraic solution if extra variables  are allowed and, furthermore, the
original analytic solution is recovered as an analytic section of this new algebraic solution.
We should remark that the above result is also true for the field of real numbers an this will be used, in Section \ref{section-realanalytic},
in order to obtain the real analytic version of Theorem \ref{theo-main}.
The above  result is a version of Artin's Approximation Theorem,
which in its original  statement says that a system of polynomial equations
has analytic \cite{Artin1968} or algebraic \cite{Artin1969}  solutions provided it has formal solutions.
To give a better picture,  it is an algebraic version of P\l oski's Approximation Theorem \cite{Ploski1974}, which
is a generalization of Artin's theorem with parameters.
A comprehensive presentation  of the subject of approximation theorems is provided in \cite{Rond2018}.

Let $f,g\in \C\{x\}$ be without common factors and consider the germ
of meromorphic function $\varphi = f/g$. We  write their decompositions in irreducible factors
$f = f_{1}^{n_{1}} \cdots f_{\alpha}^{n_{\alpha}}$ and $g = g_{1}^{m_{1}} \cdots g_{\beta}^{m_{\beta}}$,
where $n_{i},m_{i} \in \N$,
determined up to multiplication by unities in  $\C\{x\}$. 
We start by associating with $\varphi$  its  logarithmic derivative:
\begin{equation}
\label{eq-etaH}
\lambda_{\varphi} = \frac{\dd \varphi}{\varphi} = \frac{\dd f}{f} - \frac{\dd g}{g} =
 \sum_{i=1}^{\alpha} n_{i} \frac{\dd f_{i}}{f_{i}} -  \sum_{i=1}^{\beta} m_{i} \frac{\dd g_{i}}{g_{i}}.
\end{equation}
Its set of poles has reduced equation $f^{\red} g^{\red} =0$, where
$f^{\red} = f_{1}  \cdots f_{\alpha} $ and $g^{\red} = g_{1}  \cdots g_{\beta}$.
Next, we associate a holomorphic $1$-form  by cancelling  the poles of $\lambda_{\varphi}$ in the most economical way:
\begin{equation}
\label{eq-thetaH}
 \theta_{\varphi} = f^{\red}g^{\red} \lambda_{\varphi} =  g^{\red} \sum_{i=1}^{\alpha} n_{i} f_{1} \cdots \widehat{f_{i}} \cdots f_{\alpha} \dd f_{i}
- f^{\red}  \sum_{i=1}^{\beta} m_{i} g_{1} \cdots \widehat{g_{i}} \cdots g_{\beta} \dd g_{i}.
\end{equation}
Evidently $\theta_{\varphi}$ defines the same germ of  holomorphic foliation as that given by the levels of $\varphi$.
 Note that we have
\begin{equation}
\label{eq-etaH1}
\theta_{\varphi} =   \frac{f^{\red} g^{\red}}{fg} (g \dd f - f \dd g).
 \end{equation}
The \emph{singular set} of $\theta_{\varphi}$, denoted $\sing(\theta_{\varphi})$,
 may contain components of codimension one. However, it is clear from \eqref{eq-thetaH} that none of these components are in
the set of zeros or  poles of $\varphi$. Indeed, otherwise we would find that some $f_{i}$ divides $\dd f_{i}$  or some $g_{i}$ divides $\dd g_{i}$,
which is impossible.
If $h \in \C\{x\}$ is an irreducible equation of a
codimension one
component of $\sing(\theta_{\varphi})$, we say that its \emph{multiplicity} is $\ell$ if $h^{\ell}$ divides $\theta_{\varphi}$ but $h^{\ell+1}$ does not.
In this case,
\[ \theta_{\varphi} =  f^{\red}g^{\red} \lambda_{\varphi} = h^{\ell}  \theta_{\varphi}^{0},\]
where $\theta_{\varphi}^{0}$ is a holomorphic $1$-form that does not contain $h=0$ in its singular set.
By differentiating and manipulating this expression, we find
\[ \dd( f^{\red}g^{\red} ) \wedge h \theta_{\varphi}^{0}= f^{\red}g^{\red}(\ell  \dd h \wedge \theta_{\varphi}^{0} + h \dd \theta_{\varphi}^{0}), \]
which implies that $\dd h \wedge \theta_{\varphi}^{0} = 0$ over $h=0$. This means that the component $h=0$ of $\sing(\theta_{\varphi})$ is invariant by the foliation
induced by ${\theta_{\varphi}}$.
 Thus, it is contained in a level set of $\varphi$.
Based on these remarks, we can state
the following result:

\begin{prop}
\label{prop-multiplicity}
Within  the above context and notation, $h \in \C\{x\}$ is the equation of an irreducible
component of $\sing(\theta_{\varphi})$ of multiplicity $\ell$ if and only if there exists $c \in \C^{*}$ such that
$h$ has multiplicity $\ell +1$ as a factor of $f -cg$.
\end{prop}
\begin{proof}
As in the statement, denote by  $\ell$   the multiplicity of  $h$ as a component of $\sing(\theta_{\varphi})$.
Let $k$ be its multiplicity as a factor of $f -cg$, where
$c$ corresponds to the level of $\varphi$ containing $h=0$.
Take some point $p$ of $\varphi=0$,
sufficiently near $0 \in \C^{n}$, where $\varphi=0$ is smooth and the foliation $\F$ given by its levels  is regular.
Center new analytic coordinates $x = (x_{1},\ldots,x_{n})$ at $p$,   where
$h=0$ corresponds to $x_{n}=0$.
We can write $\theta_{\varphi} =  x_{n}^{\ell} \tau \dd x_{n}$, where $\tau \in \C\{x\}$ does not have $x_{n}$ as a factor.
On the other hand, at $p$, $f - cg = x_{n}^{k}\upsilon$,
where $x_{n}$ is not a factor of  $\upsilon \in \C\{x\}$. Differentiating $f/g - c = x_{n}^{k} \nu /g$, we have
\[ \theta_{\varphi} =   \frac{f^{\red} g^{\red}}{fg}{g^{2}} \dd \left( \frac{f}{g} \right)
=  \frac{f^{\red} g^{\red}}{fg}{g^{2}} \left(  k x_{n}^{k-1} \frac{\nu}{g} \dd x_{k} + x_{n}^{k} \dd \left( \frac{\nu}{g} \right) \right). \]
We are at the level $c \in \C^{*}$ of $\varphi$, so the factor outside brackets is holomorphic and non-zero at $p$.
Calculating the multiplicity of  $x_{n}$ in both sides of this equation, we   find   that  $k -1 = \ell$, or equivalently   $k= \ell + 1$, which proves the proposition.
\end{proof}

We keep working with the following objects: a meromorphic function $\varphi = f/g$, its associated holomorphic $1$-form $\theta_{\varphi}$,
having $h=0$, where $h \in \C\{x\}$,  as an irreducible   component of $\sing(\theta_{\varphi})$ of multiplicity $\ell$.
By Proposition \ref{prop-multiplicity},  there is $c \in \C$
and $s \in \C\{x\}$ not divisible by $h$ such that $f - cg = h^{\ell+1} s$.
Then, the vector of analytic functions
\[( \mathbf{f}(x) , \mathbf{g}(x),h(x),s(x)) = (f_{1}(x),\ldots, f_{\alpha}(x), g_{1}(x), \ldots, g_{\beta}(x),h(x),s(x))\]
is a solution for the polynomial equation in the variables
$u = (u_{1},\ldots,u_{\alpha}) \in \C^{\alpha}$, $v = (v_{1},\ldots,v_{\beta}) \in \C^{\beta}$ and $y,w \in \C$ given by
\[ u_{1}^{n_{1}} \cdots u_{\alpha}^{n_{\alpha}} - c v_{1}^{m_{1}} \cdots v_{\beta}^{m_{\beta}} = y^{\ell +1} w .\]
Now take $h_{1},\ldots,h_{\gamma}$ the equations of all irreducible components of $\sing(\theta_{\varphi})$.
For $i=1,\ldots,\gamma$,  denote by $\ell_{i}$ the multiplicity of $h_{i}=0$ in
 $\sing(\theta_{\varphi})$ and by $c_{i} \in \C$ its corresponding level in $\varphi$. We have thus a  vector  of analytic solutions $( \mathbf{f}(x) , \mathbf{g}(x),\mathbf{h}(x),\mathbf{s}(x))$ given by
 \[\begin{array}{l}
\mathbf{f}(x) = (f_{1}(x),\ldots, f_{\alpha}(x)) \smallskip \\
\mathbf{g}(x) = (g_{1}(x), \ldots, g_{\beta}(x) \smallskip  \\
\mathbf{h}(x) = (h_{1}(x), \cdots, h_{\gamma}(x)) \smallskip   \\
\mathbf{s}(x) = (s_{1}(x), \cdots, s_{\gamma}(x))
\end{array}\]
 for the system of $\gamma$ polynomial equations
 \begin{equation}
\label{eq-system}
\begin{cases}
u_{1}^{n_{1}} \cdots u_{\alpha}^{n_{\alpha}} - c_{1} v_{1}^{m_{1}} \cdots v_{\beta}^{m_{\beta}} = y_{1}^{\ell_{1} +1} w_{1} \medskip  \\
\ \ \ \ \ \ \ \ \ \ \cdots \medskip   \\
u_{1}^{n_{1}} \cdots u_{\alpha}^{n_{\alpha}} - c_{\gamma} v_{1}^{m_{1}} \cdots v_{\beta}^{m_{\beta}} = y_{\gamma}^{\ell_{\gamma} +1} w_{\gamma}
\end{cases}
\end{equation}
 in the variables $(u,v,y,w)$, where
$y  = (y_{1},\ldots,y_{\gamma}), w  = (w_{1},\ldots,w_{\gamma}) \in \C^{\gamma}$.

By applying  Theorem \ref{teo-ploski} to the system of equations \eqref{eq-system},
there are variables $z = (z_{1},...,z_{\sigma}) \in \C^{\sigma}$ and a vector of algebraic solutions
$(\mathbf{\widehat{f}}(x,z),\mathbf{\widehat{g}}(x,z),\mathbf{\widehat{h}}(x,z),\mathbf{\widehat{s}}(x,z) )$,
where
\smallskip
\begin{equation}
\label{eq-ploski-solution}
\begin{array}{l}
\mathbf{\widehat{f}}(x,z) = (f_{1}(x,z),\ldots, f_{\alpha}(x,z)) \smallskip \\
\mathbf{\widehat{g}}(x,z) = (g_{1}(x,z), \ldots, g_{\beta}(x,z)) \smallskip  \\
\mathbf{\widehat{h}}(x,z) = (h_{1}(x,z), \cdots, h_{\gamma}(x,z)) \smallskip   \\
\mathbf{\widehat{s}}(x,z) = (s_{1}(x,z), \cdots, s_{\gamma}(x,z))
\end{array}
\end{equation}
\smallskip
have all their entries in $\C\langle x,z \rangle$,  and a vector of analytic functions $\mathbf{z}(x) = (z_{1}(x),\ldots,z_{\sigma}(x))$ such
that $\mathbf{z}(0) = 0$ and
\[ (\mathbf{f}(x), \mathbf{g}(x), \mathbf{h}(x), \mathbf{s}(x)) =
 (\mathbf{\widehat{f}}(x,\mathbf{z}(x)), \mathbf{\widehat{g}}(x,\mathbf{z}(x)), \mathbf{\widehat{h}}(x,\mathbf{z}(x)),
 \mathbf{\widehat{s}}(x,\mathbf{z}(x)) )  . \]

We use the following result:
\begin{prop}
\label{prop-ideal} Let $\mathcal{I} = (\phi_{1}, \phi_{2})  \subset \C\{x\}$ be an ideal whose variety of zeros $V(\mathcal{I})$ has codimension
 two. Suppose that $\phi_{1}$ and $\phi_{2}$ are regular with respect to the variable $x_{n}$ with orders $n_{1}$ and $n_{2}$.
Then there exist $\delta, \varepsilon > 0$ such that whenever
$\phi_{1}', \phi_{2}' \in \C\{x\}$ are functions regular with respect to $x_{n}$, with orders $n_{1}$ and $n_{2}$,
satisfying  $|\phi_{1}' - \phi_{1}| < \varepsilon$ and  $|\phi_{2}' - \phi_{2}| < \varepsilon$ on the polydisc $\Delta(0,\delta) \subset \C^{n}$,
then
$V(\mathcal{I'})$ has codimension  two, where $\mathcal{I'} = (\phi_{1}',\phi_{2}')$.
\end{prop}
\begin{proof}   Suppose first that
$\phi_{1}, \phi_{2}  \in \C\{x'\}[x_{n}] $ are Weierstrass polynomials of degrees $n_{1}$ and $n_{2}$, where $x' =(x_{1},\ldots,x_{n-1})$.
 The hypothesis on $V(\mathcal{I})$ means that $\phi_{1}$ and $\phi_{2}$ have no common non-trivial factors.
Thus, their resultant $R(\phi_{1}, \phi_{2}) \in \C\{x'\}$ is non-zero. On the other hand,  $R(\phi_{1}, \phi_{2})$ is the determinant of
the Sylvester matrix of $\phi_{1}, \phi_{2}  \in \C\{x'\}[x_{n}] $,   whose  entries are zeros and the coefficients of these polynomials.
Thus, a small perturbation of these coefficients   will preserve the fact that the resultant is non-zero.
The proposition follows from Weiestrass preparation theorem,  noticing that, for a fixed degree, the coefficients of the Weiestrass polynomial of
$\phi \in \mathcal{O}_{n}$ depend continuously on the values of $\phi$ (see for instance \cite{shabat1992}).
\end{proof}

From this point on we consider our problem in two variables $x = (x_{1},x_{2})$.
Define the $1$-form
\begin{equation}
\label{eq-1form-isolated}
\omega_{\varphi} = \omega_{\varphi}(x) = \frac{\theta_{\varphi}(x)}{h_{1}(x)^{\ell_{1}} \cdots h_{\gamma}(x)^{\ell_{\gamma}}} .
\end{equation}
It is a holomorphic $1$-form having an isolated singularity at $0 \in \C^{2}$.

Fix, for the moment, some $k_{0} \in \mathbb{N}$ and write
 $\mb{z}(x) = \mb{z}_{k_{0}}(x) + \mb{r}_{k_{0}+1}(x)$, where  $\mb{z}_{k_{0}}$ comprises all terms of degrees up to $k_{0}$ in the Taylor series
of $\mb{z}$ and $\mb{r}_{k_{0}+1}$, the remaining terms. For $t$ in a disc   $\mathbb{D}_{1+\varepsilon} \subset \C$ of radius $1 +\varepsilon$, for
some small $\varepsilon > 0$,
set $\mb{z}(x,t) =  \mb{z}_{k_{0}}(x) + (1-t) \mb{r} _{k_{0}+1}(x)$. It is an analytic function in the variables $(x,t)$ such that $\mb{z}(x,0) = \mb{z}(x)$ and $\mb{z}(x,1) = \mb{z}_{k_{0}}(x)$ is polynomial.
Observe that, for all $t$,  $\mb{z}(x,t)$ and $\mb{z}(x)$  are tangent to each other with order at least $k_{0}$, and that order can be made
as large as we wish.

Set
\begin{equation}
\label{eq-deformation-tangent}
\begin{array}{l}
\mb{F}(x,t)  = (F_{1}(x,t),\ldots, F_{\alpha}(x,t)) := \mathbf{\widehat{f}}(x,\mb{z}(x,t)) \smallskip \\
\mb{G}(x,t)  = (G_{1}(x,t),\ldots, G_{\beta}(x,t)) := \mathbf{\widehat{g}}(x,\mb{z}(x,t))
 \smallskip  \\
\mb{H}(x,t)  = (H_{1}(x,t),\ldots, H_{\gamma}(x,t)) := \mathbf{\widehat{h}}(x,\mb{z}(x,t)) \smallskip
\end{array}
\end{equation}
We have
 $\mb{F}(x,0) =  \mb{f}(x)$,  $\mb{G}(x,0) =  \mb{g}(x)$ and  $\mb{H}(x,0) =  \mb{h}(x)$, while
 $\mb{F}(x,1)$,  $\mb{G}(x,1)$ and  $\mb{H}(x,1)$ are algebraic, since each one of their entries is the composition of an algebraic function and a polynomial map.
Observe that, for $i=1,\ldots,\alpha$, for each fixed $t$, $F_{i}(x,t)$ is   tangent to $f_{i}(x)$, with an order that grows with $k_{0}$.
In particular this gives that $F_{i}(x,t)$ is a non-unity in $\C\{x\}$. The same holds for the functions $G_{i}(x,t)$ and $H_{i}(x,t)$.

Define
\smallskip
\[ F(x,t) = F_{1}(x,t)^{n_{1}} \cdots F_{\alpha}(x,t)^{n_{\alpha}} \ \ \ \text{and}  \ \ \
 G(x,t) = G_{1}(x,t)^{m_{1}} \cdots G_{\beta}(x,t)^{m_{\beta}}. \]
Let $R(x,t) =  F(x,t)/ G(x,t)$.
Note that $$R(x,0) =  F(x,0)/ G(x,0) = f(x)/g(x) = \varphi(x)$$ and $R(x,1) =   F(x,1)/ G(x,1) = \rho_{alg}(x)$ is an algebraic
meromorphic function.

We calculate the logarithmic derivative
\[ \Lambda_{R} = \Lambda_{R}(x,t)  = \frac{\dd R(x,t)}{R(x,t)} = \frac{\dd F(x,t)}{F(x,t)} -  \frac{\dd G(x,t)}{G(x,t)},\]
followed by the cancelation of its poles
\[ \Theta_{R} = \Theta_{R}(x,t) = F(x,t)^{\red}  G(x,t)^{\red}  \Lambda_{R}(x,t),  \]
where
$F^{\red} = F_{1} \cdots F_{\alpha}$ and
$G^{\red} = G_{1}\cdots G_{\beta}$.
Following Proposition \ref{prop-multiplicity}, the $1-$form $\Theta_{R}$ can be divided by the product  $H_{1}(x,t)^{\ell_{1}} \cdots H_{\gamma}(x,t)^{\ell_{\gamma}}$  since $\mb{H}$, together with $\mb{F}$ and $\mb{G}$, provide a solution of the system of equations \eqref{eq-system}.
Thus, we can define
\begin{equation}
\label{eq-omegaR}
 \Omega_{R} = \Omega_{R}(x,t) = \frac{\Theta_{R}}{H_{1}(x,t)^{\ell_{1}} \cdots H_{\gamma}(x,t)^{\ell_{\gamma}}},
\end{equation}
which is an integrable holomorphic $1$-form.
Denote by $\Omega_{R,t} = \Omega_{R,t}(x)$ the restriction of $\Omega_{R}$ to each $t$-constant plane.
At the level $t=0$,  $\Omega_{R,0} = \omega_{\varphi}$   has an isolated singularity at $0 \in \C^{2}$ and
defines the  foliation given by the levels of $\varphi$.
At the level $t=1$, $\Omega_{R,1}$ defines a foliation corresponding to the levels of the algebraic meromorphic function
$R(x,1) = \rho_{alg}(x)$.

Notice that
$t\mapsto\Omega_{R}\left(x,t\right)$ can be seen as an analytic
deformation of $\Omega_{R}\left(x,0\right)=\omega_{\varphi}$.
Up to this moment, we have fixed $k_{0}$ in
$\mb{z}(x,t) =  \mb{z}_{k_{0}}(x) + (1-t) \mb{r}_{k_{0}+1}(x)$.
By letting this $k_{0}$ grow, we can trivialize this deformation.
Let us make a statement of this fact:

\begin{prop}
\label{prop-localtrivial}
For sufficiently large $k_{0}$, the deformation $t\mapsto\Omega_{R}\left(x,t\right)$ is locally trivial
for any $t\in\left[0,1\right]$.
\end{prop}
\begin{proof}
We start with the following lemma:
\begin{lem}
\label{lem-localtrivial}
For every sufficiently large $k_{0}$ and for any $t\in\left[0,1\right]$,
the $1$-form $\Omega_{R,t}$ has an isolated singularity at $\left(\mathbb{C}^{2},0\right)$.
\end{lem}
\begin{proof}
As remarked above, by increasing $k_{0}$, we can make the deformation
$\mb{z}(x,t)$ as tangent to $\mb{z}(x)=\mb{z}\left(x,0\right)$ as
we wish. It follows that the coefficients of $dx_{1}$ and $dx_{2}$
in $\Omega_{R,t}$ can be made as tangent and close as we wish to
those of $\Omega_{R,0}=\omega_{\varphi}$. Moreover, we can suppose
that, maybe after a linear change of coordinates of the variables
$x=(x_{1},x_{2})$, for every $t$, those coefficients are regular
with respect to $x_{2}$, with the same orders as the corresponding
coefficients of $\Omega_{R,0}=\omega_{\varphi}$. We can then apply
Proposition \ref{prop-ideal} to the ideal $\mathcal{I}$ generated by the coefficients
of $\Omega_{R,0}=\omega_{\varphi}$ and $\mathcal{I}^\prime $ generated by the coefficients
of $\Omega_{R,t}$ in order to obtain the lemma.
\end{proof}

We also need  the following result \cite[Lem. 2.1, p.149]{cerveaumattei1982} on the triviality of deformation
of foliations:
\begin{lem}
\label{lem-triviality}
Let $\omega$ be a germ of integrable holomorphic
$1$-form at $(\C^{n},0)$ with an isolated singularity at the origin.
There exists a positive integer $k$ such that, for every germ of
integrable $1$-form $\Omega$ at $(\C^{n+p},0)$ with the properties
\begin{enumerate}
\item $\Omega_{0}=\omega$ \ \ and
\item $\Omega$ coincides with $\omega$, at $(\C^{n+p},0)$, modulo the
ideal $(x)^{k}=(x_{1},\ldots,x_{n})^{k}$,
\end{enumerate}
there exists a germ of holomorphic diffeomorphism preserving horizontal
fibers $\Phi=\Phi(x,t)=(\phi(x,t),t)$ and a  germ of holomorphic
function $\Gamma=\Gamma(x,t)$ with $\Gamma(0,0)\neq 0$, both at $(\C^{n+p},0)$, such that
\[
\Phi_{*}\Omega=\Gamma\omega_{0}.
\]
\end{lem}
For any $t_{0}\in\left[0,1\right]$, $\Omega_{R,t_{0}}$ is an integrable
holomorphic $1-$form at $\left(\mathbb{C}^{2},0\right)$ with an
isolated singularity. Moreover, by construction, the difference
\[
\Omega_{R,t_{0}}-\Omega_{R}
\]
lies in the ideal $\left(x_1,x_2\right)^{k}$ with $k$ as big as necessary.
Therefore, Proposition \ref{prop-localtrivial} follows from Lemma \ref{lem-triviality}.
\end{proof}

In order to obtain Theorem \ref{theo-main}, it suffices to apply Proposition \ref{prop-localtrivial} a finite numbers
of times along the real interval $[0,1]$ and conclude that the germs of holomorphic foliations induced by $R(x,0) = \varphi(x)$ and by $R(x,1) = \rho_{alg}(x)$ at $(\C^{2},0)$
are holomorphically equivalent.
That is, there is a germ of biholomorphism $\Psi \in \text{Diff}(\C^{2},0)$ such that
$\varphi \circ \Psi$ and $\rho_{alg}$ define the same  germ of holomorphic foliation at $(\C^{2},0)$.
Finally, we conclude the proof of
Theorem \ref{theo-main} by showing that $\varphi \circ \Psi$  is an algebraic meromorphic function.
This is a consequence of  the lemma below.
 We recall that the field of   meromorphic first integrals
of a germ of holomorphic foliation $\F$ at $(\C^{n},0)$ has the form $\C(\kappa)$,
generated by  a meromorphic function $\kappa$ under  the left composition by   rational
functions in one variable \cite[Th.1.1, p.137]{cerveaumattei1982}. This meromorphic first integral $\kappa$ is said to be a \emph{minimal}.

\begin{lem} Let $\F$ be a germ of holomorphic foliation of codimension one at $(\C^{n},0)$ having
a meromorphic first integral. If the field of meromorphic first integrals of $\F$ contains an algebraic meromorphic function,
then all first integrals of $\F$ are algebraic.
\end{lem}
\begin{proof} Let $\rho_{alg}$ be an algebraic first integral  and let $\kappa$ be a minimal meromorphic first integral for $\F$.
Then there is a rational function $\phi(z)$ in one variable such that $\rho_{alg} = \phi \circ \kappa$.
Now, since $\rho_{alg}$ is algebraic, there exists a polynomial $P \in \C[x][t]$ such that
$P(\rho_{alg})=0$. Thus $P \circ \phi(\kappa)=0$. But $P \circ \phi$ is a rational function that annihilates $\kappa$.
Its numerator is a polynomial having the same property. We then conclude that $\kappa$ is algebraic, which evidently
implies that all elements in $\C(\kappa)$ are algebraic.
\end{proof}

\section{The real analytic case}
\label{section-realanalytic}

In this   section we shall   obtain the real variant of Theorem \ref{theo-main} from its complex version,  proved in Section \ref{sec-main-theorem}.
In order to fix  a terminology, we name \emph{real meromorphic functions} the elements of the ring of fractions of $\R\{x\}$, the ring of \emph{real analytic functions}
with real values, which are identified with convergent power series with real coefficients
in the variables $x =(x_{1},\ldots,x_{n}) \in \R^{n}$.
A real meromorphic function is \emph{algebraic} if its annihilated by a polynomial in $\R[x][t]$, where $t$ is a variable.
In order to make a distinction from the complex case, along this section,  real analytic objects will be denoted with a  circumflex sign.
%

The main tool we use is complexification, to be explained next.
Abusing notation and without risk of confusion, when referring to $\R\{x\}$ or  $\C\{x\}$, the same symbol $x$ will denote, respectively, real   or complex variables.
The \emph{complexification} of $\hat{f} \in \R\{x\}$ is the  power series    $f \in \C\{x\}$ with the same coefficients of $\hat{f}$.
If $\hat{\varphi} = \hat{f}/\hat{g}$ is real meromorphic, with  $\hat{f},\hat{g} \in \R\{x\}$, then its complexification is  $\varphi  = f /g$,
where $f$ and $g$ are, respectively, the complexifications of $\hat{f}$ and $\hat{g}$.
We will make use of the neologism \emph{decomplexification} and its variants when inverting this process.
Associate, with each $f \in \C\{x\}$,  the function $f^{*} \in \C\{x\}$ defined as $f^{*}(x) = \xbar{f(\bar{x})}$, that is,
if $f(x) = \sum_{\nu} a_{\nu} x^{\nu}$  in multi-index notation, then $f^{*}(x) = \sum_{\nu} \bar{a}_{\nu} x^{\nu}$.
Evidently, $f \in \C\{x\}$ is the complexification of an element of $\R\{x\}$ if and only if $f = f^{*}$.
In this case, we say that $f$ is \emph{real}.
For a complex meromorphic function $\varphi$,
we have the evident definition of $\varphi^{*}$ and an analogous property.
If $f \in \C\{x\}$, define
\[\re^{*}(f) = \frac{f + f^{*}}{2}   = \sum_{\nu} \re(a_{\nu}) x^{\nu} \ \ \ \text{and} \ \ \ \im^{*}(f) = \frac{f - f^{*}}{2i}
 = \sum_{\nu} \im(a_{\nu}) x^{\nu}.\]
 We have that $f = \re^{*}(f) + i \im^{*}(f)$, where $\re^{*}(f),  \im^{*}(f) \in \C\{x\}$ are series with real coefficients and, hence,
they decomplexify to $\R\{x\}$.
The complexification and decomplexification processes, as well as the $(*)$-operator, can be extended to analytic 1-forms and vector fields,
by simply considering their coefficients.

We will need the following definition:
\begin{ddef}  Let $\hat{\theta}$ be a germ of real analytic $1$-form at $(\R^{n},0)$. The \emph{algebraic singular set} of $\hat{\theta}$, denoted
$\sing_{alg}( \hat{\theta})$,  is the greatest common divisor of
its coefficients in $\R\{x\}$, determined up to the multiplication by a unity.
\end{ddef}
An  \emph{irreducible component} --- or \emph{component} for short ---
of $\sing_{alg}( \hat{\theta})$ is
some of its irreducible factors, determined up to a  unity in $\R\{x\}$. It has  an evident well defined \emph{multiplicity}.
When $n=2$,  we say that $\hat{\theta}$ has an
\emph{algebraically isolated singularity} at $0 \in \R^{2}$ if it has a singularity at this point
but   $\sing_{alg}( \hat{\theta})$ is trivial. This is equivalent to asking   its complexification $\theta$
to have an isolated singularity at $0 \in \C^{2}$.

Let us thus describe how to obtain Theorem \ref{theo-main} for $\K = \R$.
Let $\hat{\varphi} = \hat{f}/\hat{g}$ be a real meromorphic function, with  $\hat{f},\hat{g} \in \R\{x\}$ relatively prime.
Denote by $\varphi = f/g$, with $f,g \in \C\{x\}$, its corresponding complexification
and by $\hat{\theta}_{\hat{\varphi}}$ the real analytic $1$-form constructed following the steps leading to \eqref{eq-thetaH}, whose
complexification is $\theta_{\varphi}$.
The next result is a real version of Proposition \ref{prop-multiplicity}, concerning the components of the algebraic singular set:

\begin{prop}
\label{lem-real-irreducible}
An irreducible $\hat{h} \in \R\{x\}$ is a  component of $\sing_{alg}(\hat{\theta}_{\hat{\varphi}})$ of multiplicity $\ell$
if and only if $\ell$ is the largest integer for which one of the following conditions holds:
\begin{enumerate}
\item   there exist $c \in \R$ and $\hat{s} \in \R\{x\}$ such that
\[\hat{f}  - c \hat{g}  = \hat{h}^{\ell+1} \hat{s};\]
\item   there are $c \in \C \setminus \R$ and $\hat{\rho}_{1}, \hat{\rho}_{2}, \hat{\varsigma}_{1},\hat{\varsigma}_{2} \in  \R\{x\}$ such that
 $\hat{h} = \hat{\rho}_{1}^{2} + \hat{\rho}_{2}^{2}$ and
 \[ \hat{f}^{2} + |c|^{2} \hat{g}^{2}  - 2 \re(c)\hat{f}\hat{g} = (\hat{\rho}_{1}^{2} + \hat{\rho}_{2}^{2})^{\ell + 1}(\hat{\varsigma}_{1}^{2} + \hat{\varsigma}_{2}^{2}).\]
\end{enumerate}
\end{prop}
\begin{proof}
Suppose that an irreducible  $\hat{h} \in \R\{x\}$ defines a component of $\sing_{alg}(\hat{\theta}_{\hat{\varphi}})$ of multiplicity $\ell$. It    can   be of two types,
depending whether its complexification $h$  is irreducible or reducible in $\C\{x\}$.
In the irreducible case, $h$ is the equation of a component of $\sing(\theta_{\varphi})$ of multiplicity $\ell$ and, as in Proposition \ref{prop-multiplicity}, there is $c \in \R$ and $s \in \C\{x\}$ such that
\[f  - c g  = h^{\ell+1} s, \]
with $s$ non-divisible by $h$.
The left side  is invariant under the $(*)$-operator,  so is the right side  and, decomplexifying  this expression,
we obtain  (1)  in the proposition.

In the reducible case, it is easy to see that $h = \rho \rho^{*}$, where $\rho \in \C\{x\}$ irreducible. In this situation, both $\rho$ and $\rho^{*}$ have the same multiplicity
$\ell$ as the equation of an irreducible component  of $\sing(\theta_{\varphi})$, corresponding to levels $c \in \C$ and $\bar{c} \in \C$.
They satisfy the  equations,  one being the $(*)$ of the other, given by
\begin{equation}
\label{eq-algebraic_comoponent}
 f  - c g  = \rho^{\ell+1} \varsigma \ \ \ \text{and}  \ \ \  f  - \bar{c} g  = (\rho^{*})^{\ell+1} \varsigma^{*},
\end{equation}
where $\varsigma \in \C\{x\}$ does not have $\rho$ as a factor.
If $c \in \R$, then $\rho^{*}$ is a factor of $\varsigma$ in the first equation, of multiplicity $\ell+1$.
We would then have
\[ f  - c g  = \rho^{\ell+1} (\rho^{*})^{\ell+1} s = h^{\ell+1} s,\]
for some $s  \in \C\{x\}$ such that $s = s^{*}$.
This decomplexifies, also giving  (1) in the proposition's  statement.

Now, when $c \in \C \setminus \R$, the product of the equations in \eqref{eq-algebraic_comoponent}
is
\begin{equation}
\label{eq-decomplexification2}
f^{2} + c \bar{c} g^{2}  - (c + \bar{c})fg = (\rho \rho^{*})^{\ell+1} \varsigma \varsigma^{*}.
\end{equation}
Writing the decompositions    $\rho = \re^{*}(\rho) + i  \im^{*}(\rho) = \rho_{1}  + i  \rho_{2} $ and
$\varsigma =  \re^{*}(\varsigma) + i  \im^{*}(\varsigma) = \varsigma_{1} + i  \varsigma_{2}$, where   $\rho_{1},\rho_{2}, \varsigma_{1}, \varsigma_{2}$ are
the complexifications of, respectively, $\hat{\rho}_{1},\hat{\rho}_{2}, \hat{\varsigma}_{1}, \hat{\varsigma}_{2} \in \R\{x\}$,
we find that  equation \eqref{eq-decomplexification2} decomplexifies as item  (2)  in the statement.

The arguments above can be easily reversed, using the equivalence stated in Proposition \ref{prop-multiplicity},
in order to show  that conditions (1) and (2) determine a component of multiplicity $\ell$ of  $\sing_{alg}(\hat{\theta}_{\hat{\varphi}})$.
\end{proof}

Proposition \ref{lem-real-irreducible} is a necessary refinement, for the real case, of Proposition \ref{prop-multiplicity}. Indeed,   the real analytic version of the Approximation Theorem \ref{teo-ploski}  will provide a real deformation of the real form $\hat{\theta}_{\hat{\varphi}}$ which must preserve the topological structure of the \emph{complex} singular set of the complexified $1$-form ${\theta}_{{\varphi}}$.

Now we  transpose to the real case the same recipe of Section \ref{sec-main-theorem}.
Considering the equations of irreducible components  of $\sing_{alg}(\hat{\theta}_{\hat{\varphi}})$ and the
 decompositions in irreducible factors of $\hat{f}$ and $\hat{g}$ in $\R\{x\}$,  we produce a system
of real polynomial equations, similar to \eqref{eq-system}, taking into account the expressions of items  (1)  and  (2)
in Proposition \ref{lem-real-irreducible}.
We  then   apply the real analytic version of   Theorem \ref{teo-ploski}  in order to have real algebraic solutions,
as in \eqref{eq-ploski-solution},
depending on extra variables $z=(z_{1},\ldots,z_{\sigma}) \in \R^{\sigma}$, from which
we can produce a deformation of the initial solution, as tangent to the identity as we wish, as in \eqref{eq-deformation-tangent}.
We then obtain a real  meromorphic function $\hat{R}(x,t)$, where $t \in [0,1]$,
such that $\hat{R}(x,0) = \hat{\varphi}$ and $\hat{R}(x,1) = \hat{\varphi_{alg}}$ is a real algebraic meromorphic function.
The next step is to produce a real analytic $1$-form $\hat{\omega}_{\hat{\varphi}}$ as in \eqref{eq-1form-isolated}, which, by construction,
has a trivial algebraic singular set, and the corresponding real analytic deformation
$\hat{\Omega}_{\hat{R}}$, as in \eqref{eq-omegaR}.

Next
we take the complexification $R = R(x,t)$ of $\hat{R}$,  where  $(x,t)$ are now complex variables,
and then produce $\Omega_{R} = \Omega_{R}(x,t)$ as in \eqref{eq-omegaR}, which happens to be
the complexification of $\hat{\Omega}_{\hat{R}}$. Working now in dimension $n=2$,
we have that $\Omega_{R,0}  = \omega_{\varphi} $ has an isolated singularity at $0 \in \C^{2}$ and,  since
 $\Omega_{R}$ can be made as tangent to $\omega_{\varphi}$ as we wish,
by Lemma \ref{lem-localtrivial},  we can assume that   $\Omega_{R,t}$ has an isolated singularity at $0 \in \C^{2}$
for every $t \in [0,1]$.

The final step is to apply Lemma \ref{lem-triviality}  in order to obtain
the local triviality  in Proposition \ref{prop-localtrivial}. However, we have to assure
that it   can be achieved by means of a   germ of real diffeomorphism, namely
by an element in ${\rm Diff}(\C^{2},0)$ whose entries are invariant by the $(*)$-operator.
This actually follows from the Lemma's proof in \cite[p.149]{cerveaumattei1982}, which
is  essentially based on
 finding,  for $k=1,\ldots,p$,  vector fields of the form $v_{k} = \sum_{j=1}^{n} p_{kj} \frac{\partial}{\partial x_{j}} +  \frac{\partial}{\partial t_{k}}$,
 where $p_{kj} \in \C\{x,t\}$,
 all of them tangent to $\Omega$,
 that is $i_{v_{k}} \Omega  = 0$.
However, if $\Omega$ is real, i.e. $\Omega^{*} = \Omega$,
 we can assume that each $v_{k}$ is real, by eventually
 replacing it by $\re^{*}(v_{k}) = (v_{k} + v_{k}^{*})/2$ or by $\im^{*}(v_{k}) = (v_{k} - v_{k}^{*})/2i$.
The integration of these vector fields
$v_{k}$ provide  the real diffeomorphism we seek.
In order to conclude, we apply this to the real $1$-form
$\Omega = \Omega_{R}$.

\bibliographystyle{plain}
\bibliography{references}

\end{document}